\documentclass{amsart}
\usepackage{amssymb}
\usepackage{amsfonts}
\usepackage{amsmath}
\usepackage{graphicx}
\usepackage {hyperref}
\theoremstyle{plain}
\newtheorem{thm}{Theorem}[section]

\newtheorem{lemma}[thm]{Lemma}

\theoremstyle{definition}

\theoremstyle{remark}

\newtheorem{example}[thm]{Example}

\numberwithin{equation}{section}

\def\1{{\rm (1)}}
\def\2{{\rm (2)}}
\def\3{{\rm (3)}}
\def\4{{\rm (4)}}
\def\5{{\rm (5)}}
\def\a{{\rm (a)}}
\def\b{{\rm (b)}}

\begin{document}
\bibliographystyle{amsplain}

\title[ON STRONG  $(A)$-Rings]
{ON STRONG $(A)$-Rings}

\author{N. Mahdou}

\address{Department of Mathematics, Faculty of Sciences and Technology of Fez, Box 2202, University S. M. Ben Abdellah, Fez,
Maroc}

\email{mahdou@hotmail.com}

\author{A. Rahmouni Hassani}

\address{Department of Mathematics, Faculty of Sciences and Technology of Fez, Box 2202, University S. M. Ben Abdellah, Fez,
Maroc}

\email{rahmounihassani@yahoo.fr}

\thanks{}


\begin{abstract} In this paper, we introduce a strong property $(A)$
and we study the transfer of property $(A)$ and strong property
$(A)$ in trivial ring extensions and amalgamated duplication of a
ring along an ideal. We also exhibit a class of rings which satisfy
property $(A)$ and do not satisfy strong property $(A)$.

\end{abstract}

\maketitle
\section{Introduction}

Throughout this paper, all rings are commutative with identity
element, and all modules are unital. One of important properties
of commutative Notherien rings is that the annihilator of an ideal
$I$ consisting entirely of zero-divisors is nonzero (\cite{K}, p.
56). However, this result fails for some non-Notherian rings, even
if the ideal  $I$ is finitely generated (\cite{K}, p. 63). Huckaba
and Keller \cite{HK} introduced the following: a commutative ring
$R$ has property $(A)$ if every finitely generated ideal of $R$
consisting entirely of zero divisors has a non zero annihilator.
Property $(A)$ was originally studied by Quentel \cite{Q}. Quentel
used the term condition $(C)$ for property $(A)$. The class of
commutative rings with property $(A)$ is quite large; for example,
Notherian rings (\cite{K}, p. 56), rings whose prime ideals are
maximal \cite{HJ}, the polynomial ring $R[X]$ and rings whose
classical ring of quotients are Von Neumann regular \cite{HJ}.
Using this property, Hinkle and Huckaba \cite{HH} extend the
concept Kronecker function rings from integral domains to rings
with zero divisors. Many authors have studied commutative ring $R$
with property ($A$), and have obtained several results which are
useful studying commutative rings with zero-divisors. For
instance, see
\cite{AKR, H, HJ, HK, L1, L2, Q}). \\
In this paper, we investigate a particular class of rings
satisfying property $(A)$
 that we call satisfy strong property $(A)$. A ring $R$ is called satisfying
 strong property $(A)$ if every finitely generated ideal of $R$ which is
 generated by a finite number of zero - divisors elements of $R$,
  has a non zero annihilator; that is, if there exists $a_{i}\in R$
  such that $I =\sum_{i=1}^{n} R a_i$ and $a_{i}\in Z(R)$ for each $i$, then
there exists $0\neq a \in R$ such that $aI = 0$.\\
If a ring $R$ has strong property $(A)$ then $R$ has naturally
property $(A)$. Our aim in this paper is to prove that the converse is false in general. \\
Let $A$ be a ring, E an $A-$module. The trivial ring extension of
$A$ by $E$ (also called the idealization of $E$ over $A$) is the
ring $R := A\propto E$ whose underlying group is  $A\times E$ ,
the set of pairs $(a, e)$ with $a\in A$ and $e\in E$, with
multiplication
given by $(a, e)(a', e')=(aa', ae'+ a'e)$.\\
Trivial ring extensions have been studied extensively; the work is
summarized in Glaz's book \cite{G} and Huckaba's book \cite{H}.
These extensions have been useful for solving many open problems
and conjectures in both commutative and non-commutative ring
theory. See for
instance \cite{BKM, G, H, KM}. \\
Section 2 investigates the transfer of strong property $(A)$ and
property $(A)$ to trivial ring extensions. Using these results, we
construct a class of rings with property
$(A)$ which do not have strong property $(A)$. \\
The amalgamated duplication of a ring $R$ along an $R$-submodule
of the total ring of quotients $T(R)$, introduced by D'Anna and
Fontana and denoted by $R\bowtie E$ (see \cite{A, AF1, AF2}), is
the following subring of $R\times T(R)$ (endowed with the usual
componentwise operation):\\
$$R\bowtie E = \{(r, r+e)/ r\in R, e\in E\}.$$
It is obvious that, if in the $R$-module $R\oplus E$ we introduce
a multiplicative structure by setting $(r, e)(s, f):= (rs, rf + se
+ ef)$, where $r, s\in R$  and $e,f\in E$, then we get the ring
isomorphism $R\bowtie E\cong R\oplus E$. When $E^{2}=0$, this new
construction coincides with the Nagata's idealization. One main
difference between this constructions, with respect to the
idealization (or with respect to any commutative extension, in the
sense of Fossum) is that the ring $R\bowtie E$ can be a reduced
ring and it is always reduced if $R$ is a domain (see \cite{A,
AF2}). If $E= I$ is an ideal in $R$, then the ring $R\bowtie I$ is
a subring
of $R\times R$.\\
  In section 3, we examine the transfer of these
properties to the amalgamated duplication of a ring along an
ideal.
Using these results, we construct a class of ring with property $(A)$ which does not have strong property $(A)$. \\
Through this paper, we denote by an $(A)$-ring  (resp., strong
$(A)$-ring) a ring which satisfies property $(A)$
 (resp., a strong property $(A)$) and $Z(R)$ the set of all zero divisors of
 $R$.\\

\section{Strong Property $(A)$ in trivial ring extensions}\label{Pull}

We begin this section by giving an example of an $(A)$-ring which
is not a strong $(A)$-ring.\\

\begin{example}
 Let $D = K[X]$ the polynomial ring over a field $K$ and set $R := D\times D$ the direct product
 of $D$ by $D$. Then:\\
{\bf 1)} $R$ is an $(A)$-ring.\\
{\bf 2)} $R$ is not a strong $(A)$-ring.\\
\end{example}

\begin{proof} {\bf 1)}  $R := D\times D$ is an $(A)$-ring by (\cite{HKLR}, Proposition
1.3) since $D$ is an $(A)$-ring .\\
{\bf 2)} We claim that $R := D\times D$ is not a strong
$(A)$-ring. Deny. Let $I:= R(X, X) = XD\times XD = R(X, 0) + R(0,
X) $. We put $a_{1}=(X, 0)$ and $a_{2} =(0, X)$; then $a_{1}a_{2}
= 0$ and so there exists $0_{R}\neq (\alpha,\beta) \in R$ (:$=
D\times D)$ such that $(0, 0) =(\alpha, \beta) I=(\alpha X
D)\times(\beta X D)$. But, $\alpha XD = 0$ and $\beta XD = 0$
implies that $\alpha = \beta = 0$ since $D$ is a domain, a
contradiction. Therefore, $R := D\times D$ is not a strong
$(A)$-ring.
\end{proof}
\bigskip
Hong an al. \cite{HKLR} proved that a ring $R$ has property $(A)$ if
and only if the trivial ring extension $R\propto R$ has property
$(A)$. Now, we investigate the transfer of strong property $(A)$ and
property $(A)$ to the trivial ring extension of the form $R :=
A\propto
E$, where $E$ is a free $A$-module.\\

\begin{thm}\label{Pull.2}
 Let $A$ be a commutative ring, $E$ a free $A$-module and let
$R := A\propto E$. Then:\\
\1 $A$ is a strong $(A)$-ring  if and only if so is $R$.\\
\2 $A$ is an $(A)$-ring if and only if so is $R$.
\end{thm}

\begin{proof}
\1 Assume that $A$ is a strong $(A)$-ring and let
$J=\sum_{i=1}^{n} R(a_i, b_i) $ be a finitely generated ideal of
$R$ such that
$(a_{i},b_{i})\in Z(R)$ for each $i$. Two cases are then possible:\\
Case 1: $a_{i} = 0$  for all $i$.
Then for all $0\neq e\in E , (0, e)J = 0$, as desired.\\
Case 2: Assume that there exists $k$ such that $a_{k}\neq 0$ and
set $I := \sum_{i=1}^{n} Aa_{i}.$ We claim that $a_{i}$ is a
zero-divisor for all $i$. Deny. There exists $j$ such that $a_{j}$
is regular.
 Now, let $0_{R} \neq (\alpha_{j},\beta_{j})\in R$ such that
$(\alpha_{j},\beta_{j})(a_{j},b_{j}) = 0$, that is $ \alpha_{j}
a_{j} = 0 $ and $ \alpha_{j}b_{j}+ a_{j}\beta_{j}=0$ (since
$(a_{i}, b_{i})\in Z(R)$). Since $a_{j}$ is regular then $
\alpha_{j}=0$  and $a_{j}\beta_{j} = 0$. But, $\beta_{j}\in E$
which is a free $A$-module, then $\beta_{j} = \sum_{l=1}^{n}
d_{l}c_{l}$ where $C :=(c_{1},...,c_{n},...)$ is a basis of the
free $A$-module $E$ and $d_{l}\in A$ for each $l={1,..,n}$. This
implies that $a_{j}d_{l} = 0$ and so $d_{l} = 0$ for each
$l={1,..,n}$ (since $a_{j}$ is regular); this means that
$\beta_{j} = 0$ and then $(\alpha_{j},\beta_{j}) = 0_{R}$,  a
contradiction since $(\alpha_{j},\beta_{j})\neq 0_{R}$. Therefore,
$a_{i}\in Z(A)$ for each $i={1,..,n}$. \\
 Hence, there exists an element  $0\neq a\in A$ such that $aI
= 0$ since $A$ is a strong $(A)$-ring,. Let $e$ be an element of
$E$ such that $ae\neq 0$ and set $b: = (0, ae)\in R-\{0\}$. Hence,
$bJ = (0, ae)\sum_{i=1}^{n} R(a_{i},b_{i}) = \sum_{i=1}^{n}
R(0,ae)(a_{i},b_{i})=(0,0)$ since $aI = 0$.
It follows that $J$ has a nonzero-annihilator and so $R$ is a strong $(A)$-ring.\\

    Conversely, let $I=\sum_{i=1}^{n} A a_{i}$ be a finitely generated
ideal of $A$  such that $a_{i}\in Z(A)$ for each $i=1,.., n$.
Hence, there exists $ 0\neq b_{i}\in A$ such that $ b_{i}a_{i}=0$.
Set $J:=\sum_{i=1}^{n} R(a_{i},0)$ be a finitely generated ideal
of $R$. But, $(b_{i},0)(a_{i},0) = (b_{i}a_{i},0) = (0, 0)$.
Hence, there exists  $0_{R}\neq (a,e)\in R$ such that  $ (a,e)J =
0_{R}$ since $R$ is a strong $(A)$-ring; that means that $aa_{i} =
0$ and
$ea_{i}=0$ for each $i$. Two cases are then possible:\\
Case 1: $a\neq 0$. Then $aI = 0$, as desired.\\
Case 2: $a = 0$. In this case, $eI = 0$ and $e\neq 0$ (since $(0,
e)\neq (0,
 0)$).
 On the other hand, $e\in E$ is a free $A$-module, then $e$ is of the form $e=\sum_{i=1}^{n} f_{i}c_{i}$ where
 $C := (c_{1},..., c_{n},..)$ is a basis of $E$ and $f_{i}\in A$ for
 each $i = 1,...,n$. It follows that, $0 = eI = \sum_{i=1}^{n} (f_{i}I)c_{i}$ and then
 $f_{i}I = 0$ for each $i = 1,..., n$. Now, let  $j\in \{1,.., n\}$ such
 that $f_{j}\neq 0$ (possible since $e = \sum_{i=1}^{n} f_{i}c_{i}\neq
 0)$. Therefore, $f_{j}I = 0$, as desired. \\
 It follows that $I$ has a nonzero annihilator in all cases. Therefore, $A$ is a strong $(A)$-ring.
 completing the proof of $(1)$.\\

\2 Assume that $A$ is an $(A)$-ring and let $J = \sum_{i=1}^{n}
R(a_i, b_i) \subseteq Z(R)$ be
a finitely generated ideal of $R$. Two cases are then possible:\\
Case 1: $a_{i} = 0$ for all $i$. Hence, for each $0\neq e\in E$, we have $(0, e)J =(0, e)\sum_{i=1}^{n} R(0, b_i)= (0,0)$, as desired. \\
Case 2: There exists $i$ such that $a_{i}\neq 0$. Set $I :=
\sum_{i=1}^{n} Aa_{i}$. We wish to show that $I\subseteq Z(A)$.
Let $a\in I$,  that is $a=\sum_{i=1}^{n} \alpha_{i} a_{i} \in I$
for some $\alpha_{i}\in A$. Then $(a, e) := \sum_{i=1}^{n}
(\alpha_{i}, 0)(a_i, b_i)\in J \subseteq Z(R)$. Therefore, there
exists a nonzero element $(b,f)\in R$ such that $(0, 0) = (b,
f)(a,e) = (ba, be+af)$. Two cases are then possible:\\
\a:  $b\neq 0$. In this case  $ba = 0$, as desired.\\
\b:   $b = 0$. Then $f\neq 0$ (since $(b, f) \neq (0, 0)$) and $af
= 0$. But, $f\in E$ which is a free $A$-module, then
$f=\sum_{i=1}^{n} f_{i}c_{i}$ where $C := (c_{1},...,c_{n}...)$ is
a basis of the free $A$-module and  $f_{i}\in A$ for each
$i=1,..,n$. This implies that $af_{i}= 0$ for each $i = 1,...n$.
Now, let  $j\in \{1,.., n\}$ such
 that $f_{j}\neq 0$ (possible since $f = \sum_{i=1}^{n} f_{i}c_{i}\neq
 0)$. Therefore, $af_{j}=0$ and so $I\subseteq
Z(A)$. Hence, there exists a nonzero element $d\in A$ such that $0
= dI$, since $A$ is an $(A)$-ring. Let $e$ be an element of $E$
such that $de\neq 0$, and set $b := (0, de)\in R-\{0\}$ . Hence,
$bJ = (0, de)\sum_{i=1}^{n} R(a_{i},b_{i}) = (0,0)$ since $dI =
0$.
It follows that $J$ has a non-zero annihilator and so $R$ is an $(A)$-ring.\\

    Conversely, let $I=\sum_{i=1}^{n} A a_{i}\subseteq Z(A)$ be a finitely generated ideal of
$A$. Set $J :=\sum_{i=1}^{n} R(a_{i},0)$ and we wish to show that
$J\subseteq Z(R)$. Let $(b,f)\in J$, that is $(b,f) =\sum_{i=1}^{n}
(\alpha_{i}, \beta_{i})(a_{i},0) = (\sum_{i=1}^{n}\alpha_{i}a_{i},
\sum_{i=1}^{n}\beta_{i}a_{i})$ for
some $(\alpha_{i}, \beta_{i})\in R$. Two case are then possible:\\
Case 1:  $b := \sum_{i=1}^{n}\alpha_{i}a_{i} (\in I)\neq 0$. Since
$I\subseteq Z(A)$, there exists an element $0 \not= a\in A$ such
that $a(\sum_{i=1}^{n}\alpha_{i}a_{i}) = 0$. Let $e\in E$ such
that $ae\neq 0$. Then $(b,f)(0, ae)
=[\sum_{i=1}^{n}(\alpha_{i},\beta_{i})(a_{i}, 0)](0 ,ae) =
(\sum_{i=1}^{n}\alpha_{i}a_{i},
\sum_{i=1}^{n}\beta_{i}a_{i})(0,ae)=(0, 0)$, as desired. \\
Case 2: $b := \sum_{i=1}^{n}\alpha_{i}a_{i} = 0$. Then,
  $(0, f)(0, e) = (0, 0)$ for all $0 \not= e\in E$ and so $(b, f)\in Z(R)$.\\
  Hence, $(b,f)\in Z(R)$ in all cases, which means that $J\subseteq Z(R)$. \\
  Since $R$ is an $(A)$-ring, then
 $(a,e)J$ (:$= (a,e)(\sum_{i=1}^{n}A a_{i},\sum_{i=1}^{n} E a_{i})) =  (0,0)$ for some $0_{R}\neq (a, e)\in R$.
  We obtain
 $a \sum_{i=1}^{n} A a_{i} = 0$ and $a\sum_{i=1}^{n} E a_{i} + e \sum_{i=1}^{n} A
 a_{i}) = 0$, that is $aI = 0$ and $a I E + eI = 0$. Two cases are then possible:\\
$\star_{1}$: $a\neq 0$. Then $I$  has a nonzero annihilator (since
 $aI=0$), as desired.\\
$\star_{2}$: $a = 0$. In this case, $eI = 0$ and $e\neq 0$ (since
$(a, e)\neq 0$). On the other hand, $e\in E$ is a free $A$-module,
then $e$ is of
 the form $e=\sum_{i=1}^{n}b_{i} c_{i}$ where $C :=(c_{1},...,c_{n},..)$
 is a basis of a free $A$ - module $E$ and $b_{i}\in A$ for
 each $i=1,...,n$. It follows that, $0 = eI= \sum_{i=1}^{n}
 (b_{i}I)c_{i}$ and then $b_{i}I = 0$ for each $i={1,..,n}$. Now let
 $j\in {1,..,n}$ be such that $b_{j}\neq 0$ (possible since $e = \sum_{i=1}^{n} b_{i}c_{i}\neq
 0$. So, $b_{j}I = 0$ and $b_{j}\neq 0$. It follows that $I$ has a nonzero-annihilator
 and so $A$ is an $(A)$-ring.  This completes the proof of the Theorem.
\end{proof}
Now we are able to give a class of rings which are $(A)$-rings and
which are not a strong $(A)$-rings.
\bigskip
\begin{example}\label{Pull.3}
Let $A$ be an $(A)$-ring which is not a strong $(A)$-ring (see
[Example 2.1]). Set $R :=A\propto E$, where $E$ is a free  $A$ - module. Then: \\
\1  $R$ is not a strong  $(A)$-ring by Theorem 2.2.(1).\\
\2  $R$ is an $(A)$-ring by Theorem 2.2.(2).
\end{example}
\bigskip

Now we study the transfer of the strong property $(A)$ (resp.
property $(A)$) to trivial ring extensions of any ring by its
quotient field. \\

\begin{thm}\label{Pull.4}
  Let $A$ be a ring, Then:\\

\1 Let $Q(A)$ be the total ring of quotient of $A$, and let $R:=
A\propto Q(A)$ be the trivial ring extension of $A$ by $Q(A)$.
Then: \\
 {\bf a)} $R$ is a strong $(A)$-ring if and only if so is $A$.\\
{\bf b)} $R$ is an $(A)$-ring if and only if so is $A$.\\
\2 Let $A$ be an integral domain, $k :=qf(A)$ be the quotient
field of $A$, $E$ be a $k$ - vector space, and  let $R := A\propto
E$ be the trivial ring extension of $A$ by $E$. Then $R$ is a
strong $(A)$-ring. In particular, $R$ is an $(A)$-ring.\\
\3 Let $E$ be an $A$-module and let $R:= A\propto E$ be the
trivial ring extension of $A$ by $E$ such that $S^{-1}E$ =
$S^{-1}A$ for some multiplicatively closed subset $S$ of $A$
consisting of regular
elements. Then: \\
{\bf a)} $R$ is a strong $(A)$-ring if and only if so is $A$.\\
{\bf b)} $R$ is an $(A)$-ring if and only if so is $A$.
\end{thm}
\bigskip
Before proving this Theorem, we establish the following Lemmas.\\

\begin{lemma}\label{Pull.5}
Let $A$ be a ring and let $S$ be a multiplicatively closed subset
of $A$ consisting entirely of regular elements. Then  $A$ is a
strong $(A)$-ring if and only if so is $S^{-1}A$.
\end{lemma}

\begin{proof} Assume that $A$ is a strong $(A)$-ring and let $R =
S^{-1}A$. Let $J := \sum_{i=1}^{n}R a_{i}b_{i}^{-1}$  be a
finitely generated ideal of $R$ such that $a_{i}b_{i}^{-1}\in
Z(R)$. Set $I := \sum_{i=1}^{n}A a_{i}\subseteq J $. Thus, there
exists an element $0 \not= c_{i}d_{i}^{-1}\in R$ such that
$c_{i}d_{i}^{-1}a_{i}b_{i}^{-1} = 0$ (since $a_{i}b_{i}^{-1}\in
Z(R)$); hence $c_{i}a_{i} = 0$ and so $a_{i}\in Z(A)$. Since $A$
is a strong $(A)$-ring, there exists an element $0 \not= b\in A$
such that
$bI = 0$. Therefore, $bJ = 0$ and so $R$ is a strong $(A)$-ring.\\

Conversely, suppose that $R$ is a strong $(A)$-ring and let $I =
\sum_{i=1}^{n}A a_{i}$ be a finitely generated ideal of $A$ such
that $a_{i}\in Z(A)$. Then, note that $J := \sum_{i=1}^{n}R a_{i}$
is a finitely generated ideal of $R$ such that $a_{i}\in Z(R)$.
Since $R$ is a strong $(A)$-ring, there exists an element $0 \not=
ab^{-1}\in R$ such that $ab^{-1}J = 0$. Hence, $a\sum_{i=1}^{n}A
a_{i} = 0$ and so $A$ is a strong $(A)$-ring.\\
\end{proof}
\begin{lemma}\label{Pull.6}

Let $D$ be a domain, $E$ be a torsion free $D$-module and let $R
:=D\propto E$ be the trivial ring extension of $D$ by $E$. Then
$R$ is a strong $(A)$-ring. In particular, it is an $(A)$-ring.
\end{lemma}

\begin{proof} Let $J=\sum_{i=1}^{n} R(a_{i},e_{i})$ be a finitely generated proper ideal of $R$
such that $(a_{i}, e_{i})\in Z(R)$. We claim that $a_{i} = 0$ for
each $i=1,..,n$. Deny. There  exists $i =1, \ldots ,n$ such that
$a_{i} \not= 0$ and there exists $(b_{i},f_{i})\in R-(0,0)$ such
that $(0, 0) = (a_{i}, e_{i})(b_{i}, f_{i}) = (a_{i}b_{i},
a_{i}f_{i} + e_{i}b_{i})$. Thus, $a_{i}b_{i} = 0$ and $a_{i}f_{i}
+ e_{i}b_{i} =0$ and so $b_{i}=0$ (since $a_{i} \neq 0$ and $D$ is
a domain) and then $f_{i}=0$ (since $a_{i}f_{i}=0$, $a_{i} \not=
0$, and $E$ is a torsion free $D$-module), a contradiction, as
$(b_{i},f_{i})\neq (0,0)$. Hence, $a_{i}=0$ for each $i=1,..,n$.
\\
Hence, $J\subseteq 0\propto E$. Therefore, $(0, e)J\subseteq (0,
e)(0\propto E)=(0, 0)$ for each $0\neq e\in E$ and so $R$ is a
strong $(A)$-ring. In particular, $R$ is an $(A)$-ring.
\end{proof}
\begin{lemma}\label{Pull.7}(\cite{HKLR}, Proposition 2.14)
Let $A$ be a ring and let $S$ a multiplicatively closed subset of
$A$ consisting of regular elements. Then  $A$ is an $(A)$-ring if
and only if so is $S^{-1}A$.
\end{lemma}
\bigskip
\begin{proof}  [Proof of Theorem 2.4].\\

\1{\bf a)} Set $R := A\propto Q(A)$ and let $S =R\setminus Z(R)$.
 Hence, $R$ is a strong $(A)$-ring if and only if so is $S^{-1}R$,
 by Lemma 2.5 since $S =R\setminus Z(R)$. Then, $R$ is a strong
$(A)$-ring if and only if so is $Q(A)\propto Q(A)$ (:$=
S^{-1}A\propto S^{-1}A = S^{-1}R)$. On the other hand,
$Q(A)\propto Q(A)$ is a strong $(A)$-ring if and only if so is
$Q(A)$, by Theorem 2.2(1). Therefore, $R$ is a strong $(A)$-ring
if and only
if so is $(A)$, by Lemma 2.5.\\
{\bf b)} We know that $R :=A\propto Q(A)$ is an $(A)$-ring if and
only if so is $S^{-1}R$, by Lemma 2.7 since $S = R\setminus Z(R)$.
Then, $R$ is an $(A)$-ring if and only if so is $Q(A)\propto Q(A)$
(:$=S^{-1}A\propto S^{-1}A= S^{-1}R)$. On the other hand,
$Q(A)\propto Q(A)$ is an $(A)$-ring if and only if so is $Q(A)$,
by (\cite{HKLR}, Theorem 2.3). Therefore, $R$ is an
 $(A)$-ring if and only if so is $A$, by Lemma 2.7.\\

\2 Now, $R := A\propto E$ is a strong $(A)$-ring (in particular,
$R$ is an $(A)$-ring) if and only if so is $S^{-1}R$. Then $R$ is
a strong $(A)$-ring (in particular, $R$ is an $(A)$-ring) if and
only if so is $K\propto E$ (:$=S^{-1}A\propto S^{-1}E= S^{-1}R)$.
On the other hand, $K\propto E$ is a strong  $(A)$ -ring (in
particular, $K\propto E$ is an $(A)$-ring) by Lemma 2.6.
Therefore, $R$ is a strong $(A)$-ring
(in particular, $R$ is an $(A)$-ring).\\

\3{\bf a)} The ring $R:= A\propto E$ is a strong $(A)$-ring if and
only if so is $S^{-1}R$. Then $R$ is a strong $(A)$- ring if and
only if so is $S^{-1}A\propto S^{-1}A$ (:$=S^{-1}A\propto S^{-1}E
= S^{-1}R)$. On the other hand, $S^{-1}A\propto S^{-1}A$ is a
strong $(A)$-ring if and only if so is $S^{-1}A$ by Theorem
2.2(1).
Therefore, $R$ is a strong $(A)$-ring if and only if so is $A$ by Lemma 2.5.\\
{\bf b)} The ring $R:= A\propto E$ is an $(A)$-ring if and only if
so is $S^{-1}R$. Then $R$ is an $(A)$-ring if and only if so is
$S^{-1}A\propto S^{-1}A$ (:$= S^{-1}A\propto S^{-1}E = S^{-1}R)$.
On the other hand, $S^{-1}A\propto S^{-1}A$ is an $(A)$-ring if
and only if so is $S^{-1}A$ by (\cite{HKLR}, Theorem
2.3). Therefore, $R$ is an $(A)$- ring if and only if so is $A$ by Lemma 2.7.\\
 This completes the proof of Theorem 2.4.
\end{proof}
 Now, we study the homomorphic image of an $(A)$-ring (resp., a strong $(A)$-ring). First, let $A$ be an
 $(A)$- ring (resp. a strong $(A)$-ring) and $I$ an ideal of
 $A$. Then $A/I$ is not, in general, an $(A)$-ring (resp., not a strong $(A)$-ring) as the following example shows:
\begin{example}\label{Pull.9}

 Let $R$ be a ring which is not an $(A)$-ring (in particular is not a strong $(A)$-ring), $S = R[X]$
 be the polynomial ring over $R$ and set $I :=(X^{n}) = SX^{n}$ the principal ideal of $S$. Then:\\
 {\bf 1)} $S$ is a strong $(A)$-ring (in particular is an $(A)$-ring by (\cite{HK}, Corollary 1).\\
 {\bf 2)} $S/I = R[X]/(X^{n})$ is not an $(A)$-ring ( by \cite{HKLR}, Corollary 2.6), in particular it is not a strong $(A)$-ring) .
\end{example}

 Now, assume that $A/I$ be a strong $(A)$-ring (resp., an $(A)$-ring), then $A$ is not, in general, a strong $(A)$-ring (resp., an $(A)$-ring)
 even if $I^{2}=0$ as the following example shows:
\begin{example}\label{Pull.10}

 Let $K$ be a field and let $A := K[X, Y]$ be the polynomial ring over $K$
 with indeterminates $X$ and $Y$, which is an $(A)$- ring.
Let $E = \bigoplus_{p}A/(p)$ where $p$ ranges over the primes of $A$
and $(p)$ is the principal ideal of $A$ generated by $p$.
Set $R :=A\propto E$ the trivial ring extension of $A$ by $E$. Then:\\
{\bf 1)} $R$ is not an $(A)$-ring (by \cite{HKLR}, Example 2.4). In particular, $R$ is not a strong $(A)$-ring).\\
{\bf 2)} Set $I := 0\propto E$. Then $R/I (= A = k[X,Y]$) is an
$(A)$-ring. In particular,  $R/I$ is a strong $(A)$-ring).
\end{example}
\bigskip

\section{Strong property $(A)$ in amalgamated duplication of a ring along an ideal}
\label{WCTE} This section is devoted to the transfer of strong
Property $(A)$ and Property $(A)$ in amalgamated duplication of a
ring along an ideal. Recall that a regular ideal
of a ring $R$ is an ideal which contains a regular element of $R$.\\

\begin{thm}\label{WCTE.1}
 Let $R$ be a ring , $I$ be a proper ideal of $R$ and let $S :=R
\bowtie I$ be  the amalgamated duplication of a ring $R$
along $I$. Then: \\
{\bf 1)} If $R\bowtie I$ is a strong $(A)$-ring, then so is $R$.\\
{\bf 2)} {\bf a)} If $R\bowtie I$ is an $(A)$-ring, then so is $R$.\\
{\bf b)} Assume that $I$ is a regular ideal of $R$. Then $R\bowtie
I$ is $(A)$- ring if and only if so is $R$.
\end{thm}

Before proving this Theorem, we establish the following Lemma.\\

\begin{lemma}\label{WCTE.2}(\cite{AF2}, Corollary 3.3(d))
 Let $R$ be  a commutative ring, $I$ be a regular ideal of $R$ and
 $Q(R)$ be the total ring of fractions of $R$. Then $Q(R\bowtie I) =
 Q(R)\times Q(R)$.
\end{lemma}
\bigskip
\begin{proof}[Proof of Theorem 3.1]

 {\bf 1)} Let $J=\sum_{i=1}^{n} R a_{i} $ be a finitely generated ideal
of $R$ such that $a_{i}\in Z(R)$ for each $i={1,..,n}$; that is
there exists $0\neq b_{i}\in R $ such that $ b_{i}a_{i}=0$. Set
 $L := \sum_{i=1}^{n} (R\bowtie I)(a_{i}, 0)$ be a finitely generated ideal of
 $S$. But $(a_{i}, 0)\in Z(R\bowtie I)$ since $(b_{i}, 0)(a_{i}, 0) =(b_{i}a_{i}, 0) = (0,
0)$ since $b_{i}a_{i} = 0$. Hence, there exists $(0, 0)\neq (b,
j)\in S$ such that $(0, 0) = (b, j) L =(b, j)\sum_{i=1}^{n}
(R\bowtie I)(a_{i},0)=(b, j)(\sum_{i=1}^{n}\alpha_{i} a_{i},
\sum_{i=1}^{n}j_{i}a_{i})$  since $S$ is a strong $(A)$-ring. We
obtain $b \sum_{i=1}^{n} R a_{i} = 0$  and  $j
\sum_{i=1}^{n}j_{i}a_{i}= 0$. Two cases are then possible:\\
Case 1:  $b\neq 0$. Then $bJ = 0$, as desired.\\
Case 2:  $b = 0$. Then  $0\neq j\in I$ and we have
 $(0,0)=(0, j)\sum_{i=1}^{n} (R\bowtie I)(a_{i},
0)= \sum_{i=1}^{n} (R\bowtie I)(0, ja_{i})$. Hence $ja_{i}=0$ for
all $i=1,..,n$ and so $jJ=0$. \\
In all cases, $J$ has a non zero annihilator. Therefore, $R$ is a strong $(A)$-ring.\\

 {\bf 2)} {\bf a)}  Let $J:=\sum_{i=1}^{n} R a_{i} ( \subseteq Z(R))$ be a finitely generated
ideal of $R$ . Set $L := \sum_{i=1}^{n} (R\bowtie I)(a_{i}, 0)$ be
a finitely generated ideal of $S$. Our aim is to show that
$L\subseteq Z(R\bowtie I)$. Let $(b, j)\in L$, that is
$(b,j)=\sum_{i=1}^{n}(\alpha _{i}, j_{i})(a_{i},
0)=(\sum_{i=1}^{n}\alpha_{i} a_{i}, \sum_{i=1}^{n}j_{i}a_{i})$. Two cases are then possible:\\
Case 1: $b = \sum_{i=1}^{n}\alpha_{i} a_{i}\neq {0}$. Hence, there
exists $0\neq a\in {R}$ such that $a(\sum_{i=1}^{n}\alpha_{i}
a_{i})=0$. Since $J\subseteq Z(R)$,
we obtain $(b, f)(a,-a)=(0, 0)$ as desired.\\
Case 2: $b = \sum_{i=1}^{n}\alpha_{i} a_{i}=0$. Then $(0, j)(a,
-a)=
(0, 0)$ for all $0\neq a \in R$. \\
Then in all cases $(b, j)\in Z(R\bowtie I)$, and so $L\subseteq
Z(R\bowtie I)$.
 Since $S$ is an $(A)$-ring, then for some $0 \not= (b,
j)\in S$, $(0, 0)= (b, j) L $, that is $(0, 0)=(b,
j)(\sum_{i=1}^{n}\alpha_{i} a_{i}, \sum_{i=1}^{n}j_{i}a_{i})$. We
obtain $b \sum_{i=1}^{n} R a_{i} =
0$ and  $j \sum_{i=1}^{n}j_{i}a_{i}= 0$. Two cases are then possible:\\
$\star _{1}$:  $b\neq 0$. Then $bJ = 0$, as desired.\\
$\star _{2}$:  $b = 0$. Then   $0\neq j\in I$ and we have
 $(0,0)=(0, j)\sum_{i=1}^{n} (R\bowtie I)(a_{i},
0)= \sum_{i=1}^{n} (R\bowtie I)(0, ja_{i})$. Therefore, $ja_{i}=0$
for
all $i=1,..,n$ and so $jJ=0$ .\\
In all cases, $J$ has a nonzero annihilator. Hence, $R$ is an
$(A)$-ring.\\
 {\bf b)} If $R\bowtie I$ is an $(A)$-ring, then so is $R$ by (2-a).
 Conversely, assume that $R$ is an $(A)$-ring. Then $Q(R)$ is an $(A)$-ring  by (\cite{HKLR}, Proposition 2.14),
 and so
$Q(R)\times Q(R)$ is an $(A)$-ring  by (\cite{HKLR}, Proposition
1.3), that means that $Q(R\bowtie I)$  is an $(A)$- ring (by Lemma
3.2). Therefore, $R\bowtie I$ is an $(A)$-ring by (\cite{HKLR},
Proposition 2.14) and this completes the proof of Theorem 3.1.
\end{proof}
\bigskip
In general, a ring $R$ is a strong $(A)$-ring does not imply that
the amalgamated duplication of $R$ along an ideal $I$ of $R$  is a
strong $(A)$-ring even if $R$ is a local integral domain and $I$
is its maximal ideal as the following example shows. It is also a
new example
of a ring which is an $(A)$-ring and is not a strong $(A)$-ring .\\
\bigskip
\begin{example}\label{WCTE.4}
Let $K$ be a field and let $R = K[[X]]$ be the power series ring
over $K$ . Let $I =(X)$ be the ideal of $R$ generated by $X$ .
Consider $S :=R \bowtie I$ be the amalgamated duplication of a
ring $R$ along
$I$. Then:\\
 {\bf 1)} $R$ is a strong $(A)$-ring.\\
 {\bf 2)} $S:= R \bowtie I$ is not a strong $(A)$-ring.\\
 {\bf 3)} $S$ is an $(A)$-ring.
\end{example}

\begin{proof}
{\bf 1)} Clear since $R$ is an integral domain.\\
 {\bf 2)} Let $J := S(0, X) + S(X, -X) $  be a finitely
generated ideal of
 $S$, with $(0, X)(X,-X) = 0$ . We
claim that does not exist $(P,P+ Q)\in S$ such that $(P,P+Q)J =
0$. Deny. There exists $(P,P+ Q)\in S\setminus(0,0)$ such that
$(P,P+ Q)J = 0$. But $(P,P+ Q)(X, -X) = 0$ implies that $P = 0$ (
since $R$ is a domain ) and $(0, Q)(0,X) = 0$ implies that $Q =
0$, a contradiction.
Therefore, $S$ is not a strong $(A)$-ring.\\
 {\bf 3)} By Theorem 3.1(2.b)
\end{proof}
\bigskip




\begin{thebibliography}{999}\addcontentsline{toc}{section}{\protect\numberline{}{Bibliography}}

\bibitem{AKR}   F. Azarpanah, O. A. S. Karamzadeh and A. Rezai Aliabad, On ideals consisting entirely of zero divizors, Comm. Algebra 28
(2000),1061-1073.

\bibitem{BKM}   C. Bakkari, S. Kabbaj and N. Mahdou, Trivial extensions defined by Pr\^ufer conditions, J. Pure
Appl. Algebra. In Press.

\bibitem{A}    M. D'Anna, A construction of Gorenstein rings, J. Algebra. 306 (2006), no. 2, 507-519.

\bibitem{AF1}  M. D'Anna and M. Fontana, The amalgamated duplication of a ring along a multiplicative-canonical ideal, Ark. Mat. 45(2007), no. 2, 241--252.

\bibitem{AF2}  M. D'Anna and M. Fontana, An amalgamated duplication of a ring along an ideal: the basic
 properties, J. Algebra Appl. 6 (2007), no. 3, 443-459.

\bibitem{G}    S. Glaz, Commutative coherent rings, Lecture Notes in Mathematics, 1371. Springer-Verlag, Berlin, 1989.

\bibitem{HJ}   M. Henriksen and M. Jerison, The space of minimal prime ideals of a commutative ring, Trans. Amer. Math. Soc. 115 (1965), 110-130.

\bibitem{HH}   G. Hinkle and J. Huckaba, The generalized Kronecker function ring and the ring R(X), J. Reine Angew. Math. 292 (1977), 25-36.

\bibitem{HKLR} C. Y. Hong, N. K. Kim, Y. Lee and S. J. Ryu, Rings with Property $(A)$ and their extensions, J. Algebra 315 (2007), 612-628.

\bibitem{H}    J. A. Huckaba, Commutative rings with zero divisors, Marcel Dekker, New York, 1988.

\bibitem{HK}    J. A. Huckaba and J. M. Keller, Annihilation of ideals in commutative rings, Pacific J. Math. 83 (1979), 375-379.

\bibitem{K}    I. Kaplansky, Commutative Rings, Allyn and Bacon, Boston, 1970.

\bibitem{KM}   S. Kabbaj and N. Mahdou, Trivial extensions defined by coherent-like conditions, Comm. Algebra 32 (10) (2004), 3937-3953.

\bibitem{L}    J. Lambek, Lectures on Rings and Modules, Bllaisdall, Waltham, Mass.,1966.

\bibitem{L1}   T. G. Lucas, Two annihilator conditions: Property $(A)$ and (a.c), Comm. Algebra 14 (1986), 557-580.

\bibitem{L2}   T. G. Lucas, The diameter of a zero divisors graph, J. Algebra 301 (2006), 174-193.

\bibitem{MY}   H. R. Maimani and S. Yassemi, Zero divisor graphs of amalgamated duplication of a ring along an ideal,
 J. Pure Appl. Algebra 212 (2008), no. 1, 168-174.

 \bibitem{Q}   Y. Quentel, Sur la compacit\'e du spectre minimal d'un anneau, Bull. Soc. Math. France 99 (1971), 265-272.


\end{thebibliography}
\end{document}